\documentclass[a4paper, 12pt]{article}
\usepackage{anysize}
\marginsize{3,5cm}{2,5cm}{2,5cm}{2,5cm}
\usepackage{amssymb}
\usepackage{amsmath,amsbsy,amssymb,amscd}
\usepackage{t1enc}
\usepackage[cp1250]{inputenc}
\usepackage[british]{babel}
\usepackage[all]{xy}
\usepackage{color}\usepackage{amsfonts}
\usepackage{latexsym}
\usepackage{amsthm}
\usepackage{mathrsfs}
\usepackage{hyperref}
\usepackage{graphicx}

\usepackage{color}
\usepackage{caption}
\usepackage{subcaption}
\usepackage{enumerate}
\usepackage{indentfirst}

\usepackage{mathrsfs} 

\DeclareMathAlphabet{\mathpzc}{OT1}{pzc}{L}{it} 

\newtheorem{definition}{Definition}[section]

\newtheorem{proposition}[definition]{Proposition}
\newtheorem{theorem}[definition]{Theorem}
\newtheorem{corollary}[definition]{Corollary}
\newtheorem{remark}[definition]{Remark}
\newtheorem{lemma}[definition]{Lemma}

\def\R{\mathbb{R}}
\def\T{\mathbb{T}}

\def\Z{\mathbb{Z}}
\def\N{\mathbb{N}}

\def\rd{\mathrm{d}}

\def\cB{\mathcal{B}}

\def\fg{\mathfrak g}

\def\cD{\mathcal{D}}

\newcommand{\bea}{\begin{eqnarray}}
  \newcommand{\eea}{\end{eqnarray}}
  \newcommand{\beab}{\begin{eqnarray*}}
  \newcommand{\eeab}{\end{eqnarray*}}

  \newcommand{\be}{\begin{equation}}
  \newcommand{\ee}{\end{equation}}

\title{Bernoulli property for certain skew products over hyperbolic systems}
\author{Changguang Dong,\ Adam Kanigowski}
\begin{document}
\baselineskip=14pt \maketitle
\begin{abstract}
We study the Bernoulli property for a class of partially hyperbolic systems arising from skew products. More precisely, we consider a hyperbolic map $(T,M,\mu)$, where $\mu$ is a Gibbs measure, an aperiodic H\"older continuous cocycle $\phi:M\to \R$ with zero mean and a zero-entropy flow $(K_t,N,\nu)$. We then study the skew product 
$$
T_\phi(x,y)=(Tx,K_{\phi(x)}y),
$$
acting on $(M\times N,\mu \times \nu)$. We show that if $(K_t)$ is of slow growth and has good equidistribution properties, then $T_\phi$ remains Bernoulli. In particular, our main result applies to $(K_t)$ being a typical translation flow on a surface of genus $\geq 1$ or a  smooth reparametrization of isometric flows on $\T^2$. This provides examples of non-algebraic, partially hyperbolic systems which are Bernoulli and for which the center is non-isometric (in fact might be weakly mixing).
\end{abstract}
\tableofcontents

\section{Introduction}
Chaotic properties of smooth dynamical systems have been an active area of research for the last sixty years. One of the most chaotic property of a smooth system is being {\em isomorphic} to an independent process, i.e.\ being a Bernoulli system. By now, there are many examples of smooth Bernoulli systems, the main source of such coming from algebraic setting. Bernoulli property is known to hold for ergodic toral automorphisms \cite{Katznelson}, ergodic autmorphisms of nilmanifolds \cite{GorodnikSpatzier,Rud}, positive entropy translations on irreducible quotients of semi-simple Lie groups \cite{Dani,Kanigowski,OrWe}. There are also Bernoulli systems outside algebraic world: Anosov maps \cite{bow,si}, Anosov flows \cite{Ratner111}, hyperbolic billiards (with singularities) \cite{Chernov}, suspensions of Anosov maps \cite{Bunimovich,Ratner101}, certain compact group extension of Anosov diffeomorphisms \cite{ru-1}. However, apart from algebraic systems, all examples of smooth Bernoulli systems have a very strong restriction on possible behavior on the center space: the center is either trivial or {\em isometric}. The main reason is that in the non-algebraic setting, the main tool of establishing Bernoulli property is the geometric mechanism introduced in \cite{OrWe} which, by its nature, puts the aformentioned strong restrictions on the center space. 

We are interested in the Bernoulli property for (non-algebraic) partially hyperbolic systems which have a non-trivial growth on the center. Historically (see eg.\ \cite{Kalikow,kat,OrnsteinRudolphWeiss,Rud2}) a successful class of smooth systems for which various chaotic properties were studied are {\em skew products}. In the abstract setting, skew products are defined as follows: fix $T:(X,\mu)\to (X,\mu)$ (the base), $S:(Y,\nu)\to (Y,\nu)$ (the fiber) and a {\em cocycle} $\phi:X\to \Z$. Then the skew product acting on $(X\times Y,\mu\times \nu)$ is given by 
\be\label{eq:skew}
T_\phi(x,y)=(Tx,S^{\phi(x)}y).
\ee
Skew products transformations are a rich source of $K$ non-Bernoulli systems \cite{Austin,Kalikow,krv,kat}. On the other hand, in the measurable category, for $T$ being the full two-shift, some rank one $S:(Y,\nu)\to (Y,\nu)$ and some special cocycles $\phi$ (crucially depending only on the zero-th coordinate), the skew product remains Bernoulli \cite{bs-1,bs-2}. To study skew products in smooth category, it is enough to take $S=(S_t)$ to be a smooth flow and a smooth cocycle $\phi:X\to \R$ in \eqref{eq:skew}.  In this setting however, the measurable approach from \cite{bs-1,bs-2} breaks down as they crucially used the fact that the cocycle depends on one (or finitely many) coordinates which is not compatible with smoothness of $\phi$. This is one of the reasons why, as mentioned above, there are no examples of non-algebraic Bernoulli systems for which the center has non-trivial growth.

In this paper, we consider skew products $T_\phi$ of the following form:
\begin{itemize}
\item the base $T:(M,\mu)\to(M,\mu)$ is a hyperbolic diffeomorphism with $\mu$ being a Gibbs measure;
\item the cocycle $\phi:M\to \R$ is aperiodic (see Definition \ref{ap}) and of zero mean;
\item the fiber $(K_t):(N,\nu)\to(N,\nu)$ is {\em quasi-elliptic} (see Definition \ref{growth-1}) and having a regular generating partition (see Definition \ref{Gen})
\end{itemize}

Our main result (see Theorem \ref{main-1} and Corollary \ref{main-3}) is that under the above assumptions the skew product $T_\phi$ is a Bernoulli system. In Section \ref{sec:qe} we show that the above assumptions on the fiber are satisfied for $K_t$ being a typical (in the measure theoretic sense) {\em translation flow} on every surface of genus $\geq 1$ and also for smooth reparametrizations of linear flows on $\T^2$. Recall that by \cite{Avila-Forni} it follows that a typical translation flow is {\em weakly mixing}. Moreover, by \cite{Fayad,Sklover} it follows that there exist weakly mixing smooth reparametrizations of linear flows on $\T^2$. As a consequence, we provide first examples of non-algebraic smooth Bernoulli systems with non-trivial growth on the center (in fact weakly mixing). This should be contrasted with a recent result in  \cite{krv}, where the authors showed non-Bernoulliness of analogous skew products on $\T^2\times \T^2$, where $K_t$ was a toral flow with a highly degenerated fixed point (Kochergin flow). The difference is that unlike for Kochergin flows our examples exhibit a slow divergence of nearby points at all scales (slow growth).
We now pass to a more precise description of our results.  

\subsection{Main results}
\begin{definition}\label{growth-1}
We say, a flow $(K_t,N,\nu,d)$ is quasi-elliptic if there exist real positive sequences $\{a_i\}$, $\{b_i\}$, $\{\delta_i\}$ with 
$$a_i\to+\infty,\;\frac{a_i}{b_i}\to0,\;\delta_i\to 0,\;\text{ as }i\to\infty$$
and a sequence $\{N_i\subset N\}$ with $\nu(N_i)\to 1$, such that the following two conditions hold:
\begin{itemize}
\item[(\textbf{A})] for any $i\in \N$ and any pair $(y_1,y_2)\in N_i\times N_i$, there exists $t_{y_1,y_2}\in [0,a_i]$ such that $d(y_1,K_{-t_{y_1,y_2}}y_2)<\delta_i$;
\item[(\textbf{B})] $K_t$ is {\em almost continuous along the orbits}:
there exists a sequence of sets $(Z_j)$, $\nu(Z_j)\to 1$ satisfying: for every $j$ and every  $\eta>0$ there exists a $\xi_j>0$ such that for every $|t|<\xi_j$ and every $y\in Z_j$,
\begin{equation}\label{eq:css}
d(K_ty,y)<\eta;
\end{equation}
\item[(\textbf{C})] for every $\epsilon>0$ there exists $i_\epsilon,j_\epsilon$ such that $\nu(Z_{j_\epsilon})>1-\epsilon$, and for every $i\geq i_\epsilon$ 
for every $y_1,y_2$ as in $(\textbf{A})$,
$$
d\Big(K_ty_1,K_{t-t_{y_1,y_2}}y_2\Big)<\epsilon,
$$
for every $t\in [0,b_i]$ for which $K_ty_1\in Z_{j_\epsilon}$.
\end{itemize}
\end{definition}

To better understand the definition we remark that  condition $({\bf A})$ is related to good equidistribution on the sets $\{N_i\}$, and  $({\bf C})$ describes slow orbit divergence along the subsequence $\{b_i\}$.

\begin{definition}\label{Gen}
We say a flow $(K_t,N,\nu,d)$ has a {\em regular generating partition} if there exists a (finite) partition $\mathcal Q$ of $N$ and $t_0\in \R$ such that 
\begin{itemize}
\item[(1)] the automorphism $K_{t_0}$ is ergodic,
\item[(2)] $\mathcal Q$ is a generating partition for $K_{t_0}$,
\item[(3)] $\lim_{\eta\to 0}\nu(V_{\eta}\partial \mathcal Q)=0$.
\end{itemize}Here $\partial \mathcal Q$ is the union of the boundaries of atoms of $\mathcal Q$, and $V_{\eta}(\partial \mathcal Q)$ is the $\eta$ neighborhood of $\partial \mathcal Q$ (namely, all the points whose distance to $\partial\mathcal Q$ is less than $\eta$).
\end{definition}



Our main result is:
\begin{theorem}\label{main-1}
Let $(\Sigma_{A},\sigma,\mu)$ be a transitive subshift of finite type with a Gibbs measure $\mu$ and let $\phi:\Sigma_{A}\to\R$ be an aperiodic (see Definition \ref{ap}), H\"older continuous function such that  $\int_{\Sigma_A}\phi \rd\mu=0$.  Assume that $(K_t,N,\nu,d)$ is a quasi-elliptic ergodic flow with a regular generating partition. Then the skew product $(\sigma_\phi,\Sigma_A\times N,\mu\times\nu)$ is Bernoulli.
\end{theorem}

Let $T\in \text{Diff}^\infty_\mu(M)$ be a hyperbolic diffeomorphism (admitting a Markov partition) with a Gibbs measure $\mu$. It is well known (\cite{bow,parry}) that $(T,M,\mu)$ is isomorphic to a subshift of finite type with a Gibbs measure, and there is a one to one map from the space of H\"older functions on $M$ to that on the shift space. With this observation, we have

\begin{corollary}\label{main-3}
Let $T\in \text{Diff}^\infty_\mu(M)$ be a hyperbolic diffeomorphism (admitting a Markov partition) with a Gibbs measure $\mu$. Let $(K_t,N,\nu,d)$ be an ergodic, quasi-elliptic flow with zero entropy, which admits a regular generating partition. Let $\phi:M\to\R$ be an aperiodic smooth function with $\int_M\phi \rd\mu=0$. Then the skew product $(T_\phi,M\times N,\mu\times\nu)$ is Bernoulli.
\end{corollary}

In Section \ref{sec:qe} we show that the class of quasi-elliptic flows with regular generating partition includes typical translation flows and smooth reparametrizations of two-dimensional linear toral flows. As a consequence we get:

\begin{corollary} 
Let $(T,\phi)$ be as in Corollary \ref{main-3}. Then for every surface $S$ of genus $\geq 1$ and a.e.\ translation flow $K_t$ on $S$ the corresponding skew product $T_\phi$ is Bernoulli. The same holds if $K_t$ is any $C^1$ smooth reparametrization of linear flow on $\T^2$.
\end{corollary}

\textbf{Acknowledgements:} The authors are indebted to Dmitry Dolgopyat for many insightful discussions.

\section{Preliminaries}

\subsection{Subshifts of finite type}
Let $m\in \N$ and  $\mathcal{A}:=\{0,\ldots,m-1\}$. For $A=(A_{ij})\in \mathcal{M}_{m\times m}$ with $A_{ij}\in\{0,1\}$, we define
$$
\Sigma_A:=\Big\{x=(x_j)_{j\in \Z}\in \mathcal{A}^\Z\;:\; A_{x_jx_{j+1}}=1\Big\}.
$$
Let $\sigma:\Sigma_A\to \Sigma_A$, $\sigma((x_j)_{j\in \Z})=(x_{j+1})_{j\in \Z}$. For $ t,s\in \Z$, we define the {\em cylinder sets} by setting 
$$
\mathscr C[a_t,\ldots,a_{t+s}]:=\{x\in \Sigma_A\;:\; x_i=a_i\text{ for }i\in [t,t+s]\cap \Z\}.
$$
We assume that $A$ is {\em irreducible}, which implies that $\sigma:\Sigma_A\to \Sigma_A$ is {\em transitive}.

Define a metric $\cD_2$ on $\Sigma_A$ by $\cD_2(\omega_1,\omega_2):=2^{-k}$  for any $\omega_1,\omega_2\in\Sigma_A$, where $k$ is the largest positive integer such that $(\omega_1)_i=(\omega_2)_i$ for any $|i|< k$.

In this paper, we are interested in the measure preserving system $$\sigma:(\Sigma_A,\mu,\cD_2)\to (\Sigma_A,\mu,\cD_2),$$ where $\mu$ is a {\em Gibbs measure} with a H\"older potential. We refer the reader to \cite{bow1} for the more detailed definition of Gibbs measure. Here we will emphasize the main property that we will use later.

If $\mu$ is a  Gibbs measure, then it has a local product structure, namely $\mu$ is equivalent to $\mu_\omega^+\times\mu_\omega^-$ locally. More precisely, there exist a $\delta>0$ and a H\"older function $h_\omega(\cdot)$ such that if $\mathcal D_2(x,\omega)<\delta$, then $$\frac{d(\mu^+_\omega\times\mu^-_\omega)}{d\mu}(x)=h_\omega(x).$$
Let $H_{\omega_1,\omega_2}$ be the canonical holonomy map (along stable manifolds) from $\Sigma_A^+(\omega_1)$ to $\Sigma_A^+(\omega_2)$. Since $\mu$ is Gibbs, it follows that for any $\epsilon>0$, there exists a $\delta>0$ such that if $\omega_1,\omega_2$ satisfies $\mathcal D_2(\omega_1,\omega_2)<\delta$, then for any cylinder $\mathscr C\subset \Sigma_A^+(\omega_1)$,  \be\label{e-ep}\left|\frac{\mu_{\omega_1}^+(\mathscr C)}{\mu_{\omega_2}^+(H_{\omega_1,\omega_2}\mathscr C)}-1\right|\le \epsilon.\ee

\subsection{Rokhlin's disintegration and conditional measures}

Let $\Sigma^+_A:=\{x=(x_j)_{j\in \N\cup\{0\}}\in \mathcal{A}^\mathbb{N}\;:\; A_{x_jx_{j+1}}=1\}$ and $\Sigma^-_A:=\{x=(x_{-j})_{j\in \N}\in \mathcal{A}^\mathbb{N}\;:\; A_{x_jx_{j+1}}=1\}$. For $x\in \Sigma_A$ we define the {\em stable set (past)}
$$
\Sigma_A^-(x):=\{y\in \Sigma_A\;:\; y_i=x_i\text{ for } i\geq 0\}
$$
and the {\em unstable set (future)}
$$
\Sigma_A^+(x):=\{y\in \Sigma_A\;:\; y_i=x_i\text{ for } i<0\}.
$$
Note that every $x\in \Sigma_A$ can be uniquely written in the form $x=(x^-,x^+)$, where $x^-\in \Sigma^-_A$ and $x^+\in\Sigma^+_A$. 

For $i\in\{+,-\}$, let $\pi^i$ denote the canonical projection from $\Sigma_A$ to $\Sigma^i_A$ ($\pi^i(x):=x^i$). Let $\mu^i$ be the pushforward of $\mu$ by $\pi^i$ on $\Sigma^i_A$, i.e. $\mu^i=(\pi^i)_*\mu$. We remark that in general $\mu\neq\mu^-\times\mu^+$. This is one of the differences between the full shift and subshifts of finite type.

We can also denote $\Sigma_A^-(x)$ for $x\in\Sigma_A^+$ since it fixes the future coordinates, and analogously $\Sigma_A^+(x)$ for $x\in\Sigma_A^-$. In particular, $\{\Sigma_A^-(x):x\in\Sigma_A^+\}$ forms a measurable partition of $\Sigma_A$. Therefore, by Rokhlin's Theorem on disintegration of measures, there exists a measurable map $\Phi$ that maps $x\in\Sigma_A^+$ to a probability measure on $\Sigma_A^-(x)$, and satisfies for any measurable set $U\in\Sigma_A$, 
\be\label{disintegration}\mu(U)=\int_{x\in\Sigma_A^+}\Phi(x)(U\cap\Sigma_A^-(x))\rd\mu^+(x).\ee

We call $\Phi(x)$ the conditional measure of $\mu$ restricted to the stable set $\Sigma_A^-(x)$. For simplicity of notation, we denote it $\mu^-_x$ for $x\in\Sigma^+_A$. Analogously, we have another family of conditional measures $\mu^+_x$ for $x\in\Sigma^-_A$.

\subsection{Cocycles over subshifts of finite type}

We first recall the following classical result (see eg.  \cite{si} or Proposition 1.2 in \cite{PP}).
\begin{lemma}\label{refine}
For any H\"older continuous function $\phi:\Sigma_{A}\to\R$, there exist two functions $\varphi$ and $h$, such that 
\begin{itemize}
\item[(1)] $\varphi$ is H\"older continuous, and only depends on the past,
\item[(2)] $h$ is H\"older continuous,
\item[(3)] $\phi=\varphi+h\circ \sigma-h$.
\end{itemize}
\end{lemma}

We remark that analogous statement holds if we replace past by future in the above lemma. We will use the above lemma to simplify the description of atoms of $\bigvee^\infty_{i=0}T^i\mathcal R$ for any generating partition $\mathcal R$ in Lemma \ref{atom}. This is crucial in later construction of matching.

\subsection{Mixing local limit theorem}

Let $\sigma:(\Sigma_A,\mu)\to (\Sigma_A,\mu)$ be a transitive subshift of finite type (with $\mu$ being a Gibbs measure). Let $\psi\in C^\beta(\Sigma_A)$ be H{\"o}lder continuous with exponent $\beta$, and $\int_{\Sigma_A}\psi \rd\mu=0$. Assume moreover that $\psi$ is not a {\em coboundary}, i.e. there does not exist a measurable solution $h$, to 
$$
\psi(x)=h(\sigma x)-h(x).
$$

Let $S_n(\psi)(x)=\sum_{k=0}^{n-1}\psi(\sigma^k(x))$. Since $\psi$ is not a coboundary, the well known Central Limit Theorem (CLT)  asserts that 
$$\mu(\{x:\frac{S_n(x)}{\sqrt{n}}\in I\})\to \int_I \fg_\varrho(t)\rd t\quad\text{as}\; n\to\infty.$$
Here $\fg_\varrho(t)=e^{-\frac{1}{2}\varrho^2t^2}$ is the Gaussian density function with deviation $\varrho$. The relation of $\psi$ and $\varrho$ is given by the  Green-Kubo formula,$$\varrho^2=\int \psi^2\rd\mu+2\sum^\infty_{n=1}\int \psi (\psi\circ \sigma^n) \rd\mu.$$

We will need a refined version of CLT, together with the mixing property. For this purpose, we introduce the following definition.
\begin{definition}\label{ap}
A function $f:\Sigma_A\to \R$ is {\bf periodic} if there exist $\rho\in\R$, $g:\Sigma_A\to\R$ measurable, $\lambda>0$ and $q:X\to \Z$, such that $f=\rho+g-g\circ \sigma+\lambda q$ almost everywhere. Otherwise, it is {\bf aperiodic}.
\end{definition}
\begin{remark}We recall that by \cite{PP}, if a H\"older cocycle $f$ is periodic, then the transfer function might be chosen to be H\"older continuous.
\end{remark}

We adapt the definition in \cite{d-n} to our case as follows.
\begin{definition}[\cite{d-n}] 
Let $T:(X,\mu)\to (X,\mu)$ and let $\psi:X\to\R$ be square integrable.
We  say  that $(T,\psi)$ satisfies  the {\bf mixing local limit theorem} (MLLT) if there are some functions $g$ and $h$, where $h$ is bounded and $\mu$-almost everywhere continuous, such that $\psi=g-h+h\circ T$, a number $\varrho$ such  that,  as $n\to\infty$,  the  following  holds:  for  any  bounded  and continuous $\alpha,\beta:X\to\R$ and  for  any  continuous  and  compactly  supported $\gamma:\R\to\R$, for any sequence $\omega_n$ satisfying
$$|\omega_n-\omega\sqrt{n}|\le K,$$ we have \be\sqrt{n}\int_X \alpha(x)\beta(T^nx)\gamma[S_n(\psi)(x)-n\mu(\psi)-\omega_n]\rd\mu(x)\to \fg_\varrho(\omega)\mu(\alpha)\mu(\beta)\int\gamma \rd u .\ee
The convergence is uniform once $K$ is fixed and $\omega$ is chosen from a compact set. Here $u$ is a measure on $\R$.
\end{definition}

We also want to remark that MLLT holds for Anosov diffeomorphisms (\cite{GH}), certain suspension flow over hyperbolic systems (\cite{d-n}), certain systems admitting Young towers (\cite{go}). We will need MLLT for transitive subshifts of finite type (for the proof see eg. \cite{GH}).

\begin{theorem}\label{sft-m}
Let $\sigma_A$ be a transitive subshift of finite type with a Gibbs measure $\mu$. Then for every aperiodic, H\"older continuous function $\psi$, $(\sigma_A,\psi)$ satisfies MLLT with the measure $u$  equivalent to Lebesgue measure on $\R$.
\end{theorem}

 The following result is essential in next section.
\begin{corollary}\label{mcharacter}
Let $(\sigma_A,\psi,\mu)$ with $\psi$ aperiodic and $\int_{\Sigma_A}\psi \rd\mu=0$. Let $\mathscr C_1,\mathscr C_2\in\Sigma_A$ be two cylinder sets. Then  for any $k\in \R$ and any compact interval $I\subset\R$, we have
\be
\sqrt{n}\,\mu(\{x\in \mathscr C_1, \sigma_A^n x\in \mathscr C_2:S_n(\psi)(x)\in k \sqrt{n}+I\})\to \mu(\mathscr C_1)\mu(\mathscr C_2)\fg_\varrho(k)u(I),
\ee
as $n\to\infty$. Moreover the convergence is uniform if $k$ is taken from a compact set.
\end{corollary}

\begin{proof}
This follows directly from the definition of MLLT and Theorem \ref{sft-m}. Indeed, since $(\sigma_A,\psi)$ satisfies MLLT, one can approximate the characteristic functions $\chi_{\mathscr C_1}$, $\chi_{\mathscr C_2}$ and $\gamma_I=\chi_{I}$ by a sequence of continuous functions. 
\end{proof}

\section{CLT and MLLT for conditional measures}

\begin{proposition}\label{mllt-gibbs-unstable}
Let $(\sigma_A,\Sigma_A,\mu)$ be a transitive subshift of finite type  with a Gibbs measure $\mu$. Assume $\psi$ is aperiodic, depends only on the past and $\int_{\Sigma_A}\psi(x)\rd\mu=0$. Then for $\mu$ a.e. $x$ the following holds: for any cylinder set $\mathscr C\subset \Sigma_A$ for any $k\in \R$ and any compact interval $I\subset\R$,
$$\sqrt{n}\mu_x^+(\{\omega\in\Sigma_A^+(x):\sigma_A^n(\omega)\in \mathscr C,S_n(\psi)(\omega)\in k\sqrt{n}+I\})\to\mu(\mathscr C)\fg_\varrho(k)u(I),$$
 as $n\to\infty$. Moreover the convergence is uniform if $k$ is taken from a compact set.
\end{proposition}
\begin{proof}
Let $I=[a,b]$, and fix small $\epsilon>0$. For any $x\in\Sigma_A$, let $\mathscr C_\ell(x):=\mathscr C[x_{-\ell},\cdots,x_0]$. For the $\epsilon$,  it is easy to see that there exists an $\ell\in \N$ such that \eqref{e-ep} holds for any $\omega_2\in\mathscr C_\ell(\omega_1)$, and also if $\omega,\hat\omega\in \mathscr C_\ell(x)$ and they have the same future, then \be\label{eq:past-1}\sum_{i=0}^\infty |\psi(\sigma_A^i\omega)-\psi(\sigma_A^i\hat\omega)|\le \epsilon/10.\ee Indeed, the existence is guaranteed by the assumption that $\phi$ is H\"older and depends only on the past. 

Denote $W_{n,k}(x,[a,b]):=\{\omega\in\Sigma_A^+(x):\sigma_A^n(\omega)\in \mathscr C,S_n(\psi)(\omega)\in k\sqrt{n}+[a,b]\}$. For any $\omega_1\in\mathscr C_\ell(x)$, we claim that $$W_{n,k}(\omega_1,[a+\epsilon,b-\epsilon])\subset H_{x,\omega_1}(W_{n,k}(x,[a,b]))\subset W_{n,k}(\omega_1,[a-\epsilon,b+\epsilon]).$$Indeed, this follows easily from the choice of $\ell$, the definition of the holonomy map $H_{\cdot,\cdot}$ and \eqref{eq:past-1}.

Therefore, by \eqref{e-ep} and \eqref{eq:past-1}, for any $\omega_1\in\mathscr C_\ell(x)$, $$(1-\epsilon)\mu_{\omega_1}^+(W_{n,k}(\omega_1,[a+\epsilon,b-\epsilon]))\le \mu_x^+(W_{n,k}(x,[a,b]))$$\be\label{in-e} \le (1+\epsilon)\mu_{\omega_1}^+(W_{n,k}(\omega_1,[a-\epsilon,b+\epsilon])).\ee Now let $\widehat{ W}_{n,k}([a,b]):=\{\omega\in\mathscr C_\ell(x):\sigma_A^n(\omega)\in \mathscr C,S_n(\psi)(\omega)\in k\sqrt{n}+[a,b]\}$. By integrating both sides of \eqref{in-e} for $\omega_1\in\mathscr C_\ell(x)$, we have $$(1-\epsilon)\mu(\widehat{ W}_{n,k}([a+\epsilon,b-\epsilon]))\le \mu_x^+(W_{n,k}(x,[a,b]))\mu(\mathscr C_\ell(x)) \le (1+\epsilon)\mu(\widehat W_{n,k}([a-\epsilon,b+\epsilon])).$$

Now notice that by Corollary \ref{mcharacter}, as $n\to\infty$, $$\sqrt{n}\mu(\widehat{W}_{n,k}([a-\epsilon,b+\epsilon]))\to \mu(\mathscr C_\ell(x))\mu(\mathscr C)\fg(k)u([a,b]+O(\epsilon)).$$Hence from this, by letting $\epsilon\to0$ ($\ell\to\infty$), we have $$\sqrt{n}\mu_x^+(W_{n,k}(x,[a,b]))\to \mu(\mathscr C)\fg(k)u([a,b]).$$
This finishes the proof.
\end{proof}

\begin{proposition}\label{clt-gibbs-unstable}
Let $(\sigma_A,\Sigma_A,\mu)$ be a transitive subshift of finite type associated with a Gibbs measure $\mu$. Assume $\psi$ is aperiodic, depends only on the past and $\int_{\Sigma_A}\psi(x)\rd\mu=0$. Then for a.e. $x$, for any interval $I\subset\R$, as $n\to\infty$
$$\mu_x^+(\{\omega\in\Sigma_A^+(x):\frac{S_n(\omega)}{\sqrt n}\in I\})\to\int_I\fg_\varrho(t)\rd t.$$
\end{proposition}
\begin{proof}
This is an immediate consequence of Proposition \ref{mllt-gibbs-unstable}.
\end{proof}


\section{Very weak Bernoulli property}

Let $T:(X,\cB,\mu)\to (X,\cB,\mu)$ and let $\mathcal{P}:=\{P_1,\ldots P_k\}$, $\mathcal{Q}:=\{Q_1,\ldots,Q_\ell\}$ be two partitions of $X$. Let $\mathcal{P}\vee\mathcal{Q}$ denote the least common refinement of $\mathcal{P}$ and $\mathcal{Q}$, i.e.
$$
\mathcal{P}\vee\mathcal{Q}=\{P_i\cap Q_j\;:\; i\leq k,j\leq \ell\}.
$$

We say that a partition $\mathcal{R}$ is {\em generating} if $\bigvee_{-\infty}^{+\infty} T^i\mathcal{R}=\cB$. We will say that a property is satisfied for $\epsilon$ a.e. atom of a partition $\mathcal P$, if the measure of the union of all atoms which don't satisfy this property is less than $\epsilon$. For two partitions $\mathcal P=(P_1,\ldots P_k)$ and $\mathcal Q=(Q_1,\ldots Q_k)$, let 
$$
\bar{d}(\mathcal{P},\mathcal{Q})=\sum_{i=1}^k\mu(P_i\triangle Q_i).
$$
The definition of $\bar{d}$ naturally extends to partitions on different spaces (see eg. (4)-(7) in \cite{OrWe}). For two sequences of partitions $\{\xi_i\}_{i=1}^n$ and $\{\eta_i\}_{i=1}^n$, define $$\bar d\Big(\{\xi_i\}_{i=1}^n,\{\eta_i\}_{i=1}^n\Big)=\frac{1}{n}\sum_{i=1}^n\bar d(\xi_i,\eta_i).$$

For a partition $\mathcal Q$ and a set $A\subset X$ let 
$$
\mathcal{Q}|A:=\{Q_i\cap A, i=1,\ldots, k\}
$$
be a partition of $A$ with the normalized measure $\mu_A$.
Our strategy of showing that a transformation is Bernoulli is to establish the very weak Bernoulli property with respect to a generating partition $\mathcal R$.
\begin{definition}A transformation $T:(X,\cB,\mu)\to (X,\cB,\mu)$ is {\bf very weak Bernoulli} {\em({\bf VWB})} if there exists a generating partition $\mathcal{R}$ satisfying the following: for every $\epsilon>0$ there exists $n>0$ such that for all $m\in \N$,  $\epsilon$ a.e. atom $r$ of $\bigvee_{i=0}^mT^i\mathcal R$ satisfies
\begin{equation}\label{eq:con}
\bar{d}\Big(\{T^{-i}\mathcal R\}_{i=1}^n,\{T^{-i}\mathcal R|r\}_{i=1}^n\Big)<\epsilon,
\end{equation}
i.e. the unconditional distribution is $\epsilon$ close to the distribution conditioned on $r$.
\end{definition}

We will use a different characterization of VWB systems. For this we recall the definition of $\epsilon$-measure preserving map.
 A map $\theta:(X,\mu)\to (Y,\nu)$ is called {\em $\epsilon$-measure preserving} if there exists a set $E'\subset  X$, $\mu(E')<\epsilon$ and such that for every $A\in X\setminus E'$, we have
$$
\left|\frac{\nu(\theta(A))}{\mu(A)}-1\right|<\epsilon.
$$
We have the following lemma:

\begin{lemma}\label{def-vwb}
A transformation $T:(X,\cB,\mu)\to (X,\cB,\mu)$ is VWB, if for every $\epsilon>0$, there exists some $n\in \N$ and a measurable set $\displaystyle{G\subset \bigvee^\infty_{i=0}T^i\mathcal R}$ (meaning it is measurable with respect to this partition) such that $\mu(G)>1-\epsilon$ and for every pair of atoms $r,\bar r\in G$, there is an $\epsilon$-measure preserving map $\Phi_{r,\bar r}:(r,\mu^r)\to (\bar r,\mu^{\bar r})$ and a set $L\subset r$ such that:
\begin{itemize}
\item[(1)] $\mu^r(L)<\epsilon$, here $\mu^r$ is the conditional measure of $\mu$ with respect to $\displaystyle{\bigvee^\infty_{i=0}T^i\mathcal R}$,
\item[(2)] If $x\notin L$, $x\in r$, then $\#\{i\in [1,n]:\Phi(x)_i=x_i\}\ge (1-\epsilon)n$.
\end{itemize}
\end{lemma}
\begin{proof}The proof is almost identical to an analogous statement (for weak Bernoulli partitions) in \cite{Shie} (in particular see the reasoning in (1)-(5) in \cite{Shie}). The only difference is that we require $\Phi_{r,\bar{r}}$ to be only $\epsilon$-measure preserving (and in \cite{Shie} the map is measure-preserving). 
Fix $\epsilon>0$. Notice that \eqref{eq:con} is equivalent to existence of a set $G_m\subset \bigvee_{i=0}^mT^i\mathcal R$ with $\mu(G_m)>1-\epsilon$ and such that for every $r,\bar{r}\in G_m$, 
\begin{equation}\label{zzzz}
\bar{d}\Big(\{T^{-i}\mathcal R|r\}_{i=1}^n,\{T^{-i}\mathcal R|\bar{r}\}_{i=1}^n\Big)<\epsilon,
\end{equation}
(see (2) in \cite{Shie}). This, by Lemma 1.3. in  \cite{OrWe} will follow if we show existence of a $\epsilon/16$-measure preserving map $\Phi=\Phi_{n,r,\bar{r}}:\{T^{-i}\mathcal R|r\}_{i=1}^n\to \{T^{-i}\mathcal R|\bar{r}\}_{i=1}^n$ (with normalized measures) and a set $B\subset  \{T^{-i}\mathcal R|r\}_{i=1}^n$ with $\mu(B)<\epsilon/16$ such that 
$$
T^i\Phi(x)\text{ and } T^ix \text{ are in the same atom of  }\mathcal R, \text{ for }(1-\epsilon)n\text{ proportion of }i\in [1,n].
$$
This is the only difference with the reasoning in \cite{Shie}, where it is said that \eqref{zzzz} is equivalent to the existence of a measure preserving map.  Now the proof is identical to the proof in \cite{Shie} (see (2)-(4)). More precisely, one uses existence of regular probabilities to define 
$$
\{T^{-i}\mathcal R|r\}_{i=1}^n, \;\; r\in \bigvee_0^{\infty}T^i\mathcal R,
$$
and next we use martingale theorem to get an integer $M$ and a set   $G\subset \bigvee_0^{\infty}T^i\mathcal R$, $\mu(G)>1-\epsilon$ and so that for $m\geq M$, $r\in G$ which is a subset of some $r_m\in \bigvee_{0}^{m}T^i\mathcal R$, we have
$$
\bar{d}\Big(\{T^{-i}\mathcal R|r\}_{i=1}^n,\{T^{-i}\mathcal R|r_m\}_{i=1}^n\Big)<\epsilon.
$$
This finishes the proof.
\end{proof}

\subsection{VWB property for skew products}

Let $\sigma:(\Sigma_A,\mu)\to (\Sigma_A,\mu)$ be a transitive subshift of finite type, $\phi:\Sigma_A\to \R$ be aperiodic in $C^\beta(\Sigma_A)$, with $\int \phi \rd\mu=0$.  Let moreover $K_t:(N,\nu,d)\to (N,\nu,d)$ be an ergodic flow of {\em zero entropy}. Let  $T=T_\phi:(\Sigma_A\times N,\mu\times \nu)\to (\Sigma_A\times N,\mu\times \nu)$ be the related skew product, $T(x,y):=(\sigma(x),K_{\phi(x)}y)$. By changing $\phi$ by a coboundary (see Lemma \ref{refine}), we may assume further $\phi$ depends only on the past. We emphasize that, $T$ is ergodic, eg. \cite{CoFoSi}.
Let $\mathcal{P}_m:=(P_0,\ldots, P_{m-1})$ be the partition of $\Sigma_A$ into cylinders $\mathscr C[\omega_0]$, and let $\mathcal Q$ be a regular\footnote{Recall that this means that $\nu(V_{\eta}(\partial \mathcal Q))\to 0$ as $\eta\to 0$.} partition of $N$. Let $\mathcal R:=\mathcal P_m\times \mathcal Q$ be the partition of $\Sigma_A\times N$. 
\begin{lemma}\label{generate}
If $K_{t_0}$ is ergodic, and $\mathcal Q$ is a generating regular partition (for $K_{t_0}$) on $N$ then $\mathcal P_m\times\mathcal Q$ is a generating partition for $T$.
\end{lemma}
To prove the above lemma, we need the following result:
\begin{lemma}If $\phi$ is aperiodic, then for $\mu$ a.e $x\in \Sigma_A$, 
$$
\{S_n(\phi)(x)\}_{n\in \Z},
$$
is dense in $\R$.
\end{lemma}
\begin{proof}Recall, \cite{Aar,Schmi}, that a number $a\in \R$ is called an {\em essential value} of the cocycle $\phi$ if for every measurable $C\in \Sigma_A$ with $\mu(C)>0$, and every $\epsilon>0$ there exists $N$ such that 
$$
\mu\Big(C\cap (\sigma_A^{-N}C)\cap\{x\in \Sigma_A\;:\; |S_N(\phi)(x)-a|<\epsilon\}\Big)>0. 
$$
Let $E(\phi)$ denote the set of essential values. Assume that for a positive measure set $A\subset X$, $\{S_n(\phi)(x)\}_{n\in \Z}$ is not dense. This means that there exists a number $a\in \R$, which is not an essential value of $\phi$. Since the set of essential values is a closed subgroup of $\R$ (see \cite{Aar,Schmi})  it follows that $E(\phi)=b\Z$ for some $b\in \R$. This however is a contradiction with the aperiodicity assumption of $\phi$ (see eg. Corollary 8.3.5. in \cite{Aar}).
\end{proof}
\begin{proof}[Proof of Lemma \ref{generate}]
Since $\mathcal Q$ is regular, it follows that $\nu(\partial \mathcal Q)=0$. Therefore, for some $Z\subset N$, $\nu(Z)=1$ and any $y,y'\in Z$, there exists $r_0\in \Z$ such that $K_{r_0t_0}(y)$ and $K_{r_0t_0}(y')$ are not in one atom of $\mathcal Q$ and moreover, $K_{r_0t_0}(y), K_{r_0t_0}(y')\notin \partial \mathcal Q$.
Notice first that if for every $n\in \Z$, $T^n(x,y)$ and $T^n(x',y')$ are in the same atom of $\mathcal{R}$, then $x=x'$ (since $\mathcal P_m$ is generating for $\sigma_A$).  We will show that $y=y'$ (for $y,y'$ in a full measure set). If not, then for every $n\in \Z$, $K_{S_n(\phi)(x)}(y)$ and $K_{S_n(\phi)(x)}(y')$ are in one atom of $\mathcal Q$. Let $x$ be such that $\{S_n(\phi)(x)\}_{n\in \Z}$ is dense in $\R$ (this holds for a full measure set of $x\in \Sigma_A$). Let $(n_i)$ be such that $S_{n_i}(\phi)(x)\to r_0t_0$. Then for a full measue set $Z'\in N$, and $w\in \{y,y'\}$, $K_{S_{n_i}(\phi)(x)}(w)\to K_{r_0t_0}(w)$. But since $K_{r_0t_0}(y), K_{r_0t_0}(y')\notin \partial \mathcal Q$, it follows that for large enough $i$, we have that $K_{S_{n_i}(\phi)(x)}(y)$ and $K_{S_{n_i}(\phi)(x)}(y')$ are in different atoms of $\mathcal Q$. This contradiction finishes the proof.
\end{proof}

We have the following result adapted to our setting: 
\begin{lemma}\label{atom}
For $\mu$ almost every atom $r\in\bigvee^\infty_{i=0}
T^i\mathcal R$, there exist a $y_0\in N$ and a $s\in\Sigma_A$ such that $(\mu\times\nu)^r$ almost everywhere,
$$r=\{(x,y):y=y_0,x\in\Sigma_A^+(s)\}=\Sigma_A^+(s)\times\{y_0\}.$$
\end{lemma}
\begin{proof}
This follows from Lemma 3.3 in \cite{krv} and the fact that $\phi$ only depends on the past.
\end{proof}
If an atom $r=\Sigma_A^+(x)\times\{z\}$, then the conditional measure on $r$ is $\mu^+_x\times\delta_{z}$, where $\delta_z$ is the Dirac measure at the point $z$. To simplify the notation, we will use $\mu^+_{x,z}$ instead of $\mu^+_x\times\delta_{z}$.
The following proposition is important in proving the VWB property.
\begin{proposition}[VWB property]\label{vwb}
Assume $T$ is the skew product with the fiber $K_t$ of zero entropy and that $\mathcal Q$ is a finite partition of $N$ such that $\mathcal P_m\times \mathcal Q$ is generating for $T$. Then $T$ is very weak Bernoulli with respect to $\mathcal P_m\times \mathcal Q$ if and only if for every $\epsilon>0$, there exist a $\hat n\in\N$ and a measurable set $G\subset \Sigma_A\times N$ with $\mu\times \nu(G)>1-\epsilon$ such that if $(x,z),(\bar x,\bar z)\in G$, there exists an $\epsilon$-measure preserving map $\Phi_{(x,z),(\bar x.\bar z)}:(\Sigma ^+_A(x)\times\{z\},\mu^+_{x,z})\to (\Sigma ^+_A(\bar x)\times\{\bar z\},\mu^+_{\bar x,\bar z})$  and a set $U\subset \Sigma^+_A(x)\times\{z\}$ such that:
\begin{itemize}
\item[(1)] $\mu^+_{x,z}(U)>1-\epsilon$,
\item[(2)] $\#\{i\in[1,\hat n]:T^i(x^-,y,z)\text{ and }T^i(\bar x^-,\Phi_{(x,z),(\bar x,\bar z)}(y,z))\text{ are in the same }\mathcal P_m\times \mathcal Q \text{ atom}\}\ge (1-\epsilon)\hat n$ if $(y,z)\in U$.
\end{itemize}
\end{proposition}
\begin{proof}
The proof is similar to Proposition 3.2 in \cite{krv}. By Lemma \ref{atom}, there exists a full measure set such that each atom in $\bigvee^\infty_{i=0}T^i\mathcal R$ is of the form $\Sigma_A^+(\cdot)\times\{\cdot\}$. Then the proposition follows from the Lemma \ref{def-vwb}.
\end{proof}


\section{Proof of Theorem \ref{main-1}}

We will show that under the assumptions of Theorem \ref{main-1} the assumptions of Proposition  \ref{vwb} are satisfied. Since the proof is quite involved, we split it into several parts.

\subsection{Summary of notations and setup}
Fix a subshift of finite type $(\Sigma_A,\sigma,\mu,\cD_2)$, and an ergodic flow $(N,K_t,\nu,d)$. Let $\phi:\Sigma_A\to\R$ be a H\"older function such that  $\int_M\phi \rd\mu=0$ and assume it is aperiodic. By Lemma \ref{refine}, $\phi$ is cohomologous to a H\"older function that only depending on the past, thus WLOG\footnote{Recall that if two cocycles $\phi$ and $\bar{\phi}$ are cohomologus, then the corresponding skew products are isomorphic.} assume $\phi(\omega)$ depends only on the past of $\omega$.

Let $\Sigma_A^-(x)$, $\Sigma_A^+(x)$ be the stable and unstable set at $x$ respectively. For $x\in \Sigma_A$, we will consider $x=(x^-,x^+)$ for $x^-=(\cdots,x_{-1},x_0)\in\Sigma_A^-$ and $x^+=(x_1,x_2,\cdots)\in\Sigma_A^+$. By Lemma \ref{atom}, the unstable manifold of the skew product on $\Sigma_A\times N$ is almost everywhere identified to $\Sigma_A^+(\cdot)\times\{\cdot\}$. Due to this, let $\mu_{x,z}^{+}:=\mu^{+}_{x}\times \delta_{z}$ be the conditional measure when restricted to unstable manifold (set) of $T_\phi$ on the product space $\Sigma_A\times N$.


Let $\mathcal{P}_m:=(P_0,\ldots, P_{m-1})$ be the partition of $\Sigma_A$ into cylinders $\mathscr C[\omega_0]$. Let $\mathcal Q$ be a finite regular partition of $N$, such that $\mathcal R:=\mathcal P_m\times \mathcal Q$ is generating for $T_\phi$ and $\lim_{\eta\to 0}\nu(U_{\eta}\partial \mathcal Q)=0$. This is guaranteed by our assumption on $K_t$ and Lemma \ref{generate}.


Fix a sufficiently small $\epsilon>0$. Let $c\in(0,\frac{1}{10})$ be such that $\nu(V_c(\partial \mathcal Q))<\frac{\epsilon}{100}$.   Let $a=a(c,\epsilon)>100$ be fixed such that 
\begin{itemize}
\item $10\epsilon^a\le c$, 
\item $2^{-[\epsilon^{-a}]\beta}\ll \epsilon^{4a}$, here $\beta$ is the H\"older exponent,
\item if $\mathcal D_2(x_1,x_2)\le 2^{-[\epsilon^{-a}]+1}$, then for any cylinder $\mathscr C\subset \Sigma_{A}^+(x_1)$ (see \eqref{e-ep}), 
\be\label{eq:e-a}\left|\frac{\mu_{x_1}^+(\mathscr C)}{\mu_{x_2}^+(H_{x_1,x_2}\mathscr C)}-1\right|\le \epsilon/4.\ee 
\end{itemize}

Let $c>\xi>0$ and $j:=j_{\epsilon^2}\in \N$ (in ({\bf C})) be such that for every $y\in Z:=Z_j$
\begin{equation}\label{eq:cz}
d(K_ty,y)<c/20
\end{equation}
for every $|t|<\xi$ (see \eqref{eq:css}) and $\nu(Z_j)\ge 1-\epsilon^2$. If necessary, by enlarging $a$ or choosing smaller $c$, we may assume $\xi\ge \epsilon^{3a}$.

\subsection{Cylinder sets}

To construct the matching map, we will work with certain cylinder sets. We define them as follows. For any $\omega\in\Sigma_A$, denote
$$B(\omega):=\Big\{\hat \omega\in\Sigma_A: \hat \omega_j= \omega_j,\; -[\epsilon^{-a}]\le j\le [\epsilon^{-a}]\Big\},$$
$$B_{k}(\omega,n):=\Big\{s\in\Sigma_A^+(\omega):S_n(\phi)(s)\in (k\sqrt{n}-\frac{\xi}{20},k\sqrt{n}+\frac{\xi}{20})\Big\}$$
 and 
 $$\hat B_{k}(\omega,n,\alpha):=B_{k}(\omega,n)\cap( \sigma^{-n}(B(\alpha))).$$

The following result gives the estimates on $\mu^+_{\omega}(\hat B_{k}(\omega,m,\alpha))$. 

\begin{lemma}\label{B-est}
Fix $k\in\mathbb R,\alpha\in\Sigma_A$, for almost every $\omega\in\Sigma_A$, then $$\sqrt n \mu_\omega^+(\hat B_{k}(\omega,n,\alpha))\to u(-\xi/20,\xi/20)\cdot \fg_\varrho(k) \mu(B(\alpha)),\text{ as }n\to\infty.$$ 
Moreover, the convergence is uniform for $k$ in any fixed compact subset of $\R$.
\end{lemma}
\begin{proof}
This is an immediate consequence of Proposition \ref{mllt-gibbs-unstable}.
\end{proof}

By Proposition \ref{clt-gibbs-unstable}, there exist a set $\Sigma_3\subset \Sigma_A$ with $\mu(\Sigma_3)\ge 1-\epsilon^2$, $k_0>0$ and $n_0>0$ such that for any $\omega\in\Sigma_3$, $n\ge n_0$ 

\be\label{clt-k}\mu^+_\omega(\{s\in\Sigma_A^+(\omega):|S_n(\phi)(s)|\ge k_0\sqrt n\})<\epsilon^2.\ee We may enlarge $a$ if necessary, so that $k_0\le \epsilon^{-a}$.

Let $u_\xi:=u(-\xi/20,\xi/20)$. By Lemma \ref{B-est}, there exist a set $\Sigma_4\subset\Sigma_3$ with $\mu(\Sigma_4)\ge 1-2\epsilon^2$ and $n_1\ge n_0$, such that for any $n\ge n_1$ and $k\in[-k_0,k_0]$, 

$$|\sqrt n \mu_\omega^+(\hat B_{k}(\omega,n,\alpha))- u_\xi\cdot \fg_\varrho(k) \mu(B(\alpha))|\le \epsilon^{10 a}\mu(B(\alpha)).$$

Notice that the above inequality is equivalent to 
\begin{equation}\label{B-bound}
\frac{u_\xi \cdot \fg_\varrho(k)- \epsilon^{10 a}}{\sqrt{n}}\mu(B(\alpha))\le \mu_\omega^+(\hat B_{k}(\omega,n,\alpha))\le \frac{u_\xi \cdot \fg_\varrho(k) + \epsilon^{10 a}}{\sqrt{n}}\mu(B(\alpha)).
\end{equation}

Let's remark that, by the definitions of $k_0,n_1$ (see \eqref{clt-k}), for any $\omega\in\Sigma_4$ and $n\ge n_1$, 
\begin{align}\label{eq-be}
\mu_\omega^+\Big(\bigcup_{\alpha\in\Sigma_A}\bigcup_{k\in[-k_0,k_0]}\hat B_{k}(\omega,n,\alpha)\Big)&=\mu_\omega^+(\bigcup_{\alpha\in\Sigma_A}\bigcup_{k\in[-k_0,k_0]}[B_{k}(\omega,n)\cap \sigma^{-n}(B(\alpha))])\nonumber\\&=\mu_\omega^+(\bigcup_{k\in[-k_0,k_0]}B_{k}(\omega,n))\ge 1-2\epsilon^2.
\end{align}

 \subsection{Good subsets on $\Sigma_A\times N$}

We will find ``good'' points (with respect to $T_\phi$) on the fiber $N$ as well as on the $\Sigma_A\times N$.

Since $T_\phi$ is ergodic, by Birkhoff ergodic theorem, there exist $L_1>0$ and a subset $V\subset \Sigma_A\times N$ with $\mu\times \nu(V)\ge 1-\epsilon^2$ such that for any $(x,z)\in V$ there exists a subset 
and any $L\ge L_1$,
\begin{equation}\label{E1}
\frac{1}{L}\sum_{j=0}^{L-1}\chi _{\Sigma_A\times V_c(\partial \mathcal Q)}\circ T_\phi^j (x,z)<\epsilon/20\text{, and }\frac{1}{L}\sum_{j=0}^{L-1}\chi _{\Sigma_A\times Z}\circ T_\phi^j (x,z)>1-\epsilon/20.
\end{equation}
By the definition of quasi-elliptic, there is an $i\geq i_{
\epsilon^{2a}}$ (in {\bf C}), such that \begin{equation}\label{k-e}
\nu(N_i)\ge 1-\frac{\epsilon^3}{100},\;a_i\ge 10k_0^2 n_1+100L_1+\epsilon^{-100a},\;\frac{a_i}{b_i}\le \epsilon^{1000a},\;\delta_i\le \epsilon^{10a}.
\end{equation}
Let $S_1:=V\cap (\Sigma_4\times N_i)$.

Now by Lemma \ref{atom}, there exists a full measure subset $\Gamma_0\subset \bigvee^\infty_{i=0}T^i_\phi\mathcal R$, such that if $r\in\Gamma_0$, then there exists a $z_0\in N$ and a $s\in\Sigma_A$ such that $(\mu\times\nu)^r\; a.e.$ 
$$r=\{(x,z):z=z_r,x\in\Sigma_A^+(s)\}=\Sigma_A^+(s)\times\{z_r\}.$$Let $$\Gamma_1=\{r=\Sigma_A^+(s)\times\{z_r\}\in\Gamma_0:(s,z_r)\in S_1,\mu^+_{s,z_r}(r\backslash S_1)<\epsilon/40\}.$$Since $\mu_{x,z_r}^+$ is a disintegration of $\mu\times\nu$, it follows that \be\label{gamma-1}\mu\times\nu(\Gamma_1)\ge 1-\epsilon.\ee

\subsection{Construction of the matching}

Now fix two atoms $r_1,r_2\in\Gamma_1$. By the definition of $\Gamma_1$ and Lemma \ref{atom}, there are $x,\bar x$ and $z,\bar z\in N_i$ such that $(x,z)\in S_1$ and $$r_1\overset{a.e.}{=}\Sigma_A^+(x)\times\{z\},\;\;\text{  }\;r_2\overset{a.e.}{=}\Sigma_A^+(\bar x)\times\{\bar z\}.$$

Since $z,\bar z\in N_i$, 
\begin{equation}\label{eq:zz}
d(z,K_{-\ell}\bar z)<\delta_i \;\text{ for some }\; \ell\in[0,a_i]
\end{equation}

Fix $n_2\ge n_1$ such that 
\begin{equation}\label{eq:aa}
d(\partial \mathcal Q,K_{S_{n_2}(\phi)(x)}(z))\ge c
\text{, and }\;\epsilon^{-20a}a_i\le \sqrt{n_2}\le \epsilon^{800a}b_i.
\end{equation}
This is possible because $(x,z)\in S_1$, $a_i\ge 100L_1$, and $\frac{a_i}{b_i}<\epsilon^{1000a}$.

Firstly, in order to construct the matching, we will start with the following two sets for any fixed $\alpha\in\Sigma_A$:$$\mathbf{B}_\alpha:=\bigcup_{k\in[-k_0,k_0]}\hat B_{k}(x,n_2,\alpha)\text{, and }\;{\mathbf{\bar B}}_\alpha:=\bigcup_{k\in[-k_0,k_0]}\hat B_{k}(\bar x,n_2,\alpha)$$

For any $k\in [-k_0,k_0]$, let \be\label{eq:Psi}\Psi(k):=k-\frac{\ell}{\sqrt{n_2}}.\ee Notice that by the choice of $n_2$ and $\ell\le a_i$, it follows that $\ell\le a_i\le \epsilon^{20a}\sqrt{n_2}.$ So it follows immediately that $$\Psi([-k_0+\epsilon^{10a},k_0])\subset [-k_0,k_0].$$

Let $I$ be the (finite) set of $k$'s of maximal cardinality such that, $$\mathbf{B}_\alpha=\bigcup_{k\in I}\hat B_{k}(x,n_2,\alpha)$$ and for any $k\neq k'\in I$, $$\hat B_{k}(x,n_2,\alpha)\cap \hat B_{k'}(x,n_2,\alpha)=\emptyset.$$
Since by definition, for any $y\in \hat B_{k}(x,n_2,\alpha)$, $S_{n_2}(\phi)(y)\in (k\sqrt{n_2}-\frac{\xi}{20},k\sqrt{n_2}+\frac{\xi}{20})$, it follows that $\# I\approx 20k_0\sqrt{n_2}/\xi$. 

Let
$$\tilde{\mathbf B}_\alpha:=\bigcup_{k\in\Psi (I)\cap [-k_0,k_0]}\hat B_{k}(\bar x,n_2,\alpha)\subset \mathbf{\bar B}_\alpha.$$
Since $\Psi:\Psi^{-1}(\Psi (I)\cap [-k_0,k_0])\to \Psi (I)\cap [-k_0,k_0]$ is a translation, and $\frac{\ell}{\sqrt{n_2}}\le \epsilon^{20a}$, $$\bigcup_{k\in [-k_0+\epsilon^{10a},k_0]}\hat B_{k}(\bar x,n_2,\alpha)\subset \tilde{\mathbf B}_\alpha, $$ thus it follows that $$\mathbf{\bar B}_\alpha\backslash \mathbf{ \tilde B}_\alpha\subset \bigcup_{k\in [-k_0,-k_0+\epsilon^{10a}]}\hat B_{k}(\bar x,n_2,\alpha),$$hence by Proposition \ref{mllt-gibbs-unstable} (and smoothness of the function $\fg_{\varrho}$) that $$\mu_{\bar x}^+(\tilde{\mathbf B}_\alpha)\ge \mu_{\bar x}^+({\mathbf{ \bar B}_\alpha})-\epsilon^3.$$

For any $k\in I$, we compare the measures of $\hat B_{k}(x,n_2,\alpha)$ and $\hat B_{\Psi (k)}(\bar x,n_2,\alpha)$. In fact, by triangle inequality and \eqref{B-bound}, decreasing $\xi$ if necessary to get $u_\xi<1$ (recall that $u$ is equivalent to Lebesgue),
\begin{align}\label{eee-1}
&\left|\mu_x^+(\hat B_{k}(x,n_2,\alpha))-\mu_{\bar x}^+(\hat B_{\Psi (k)}(\bar x,n_2,\alpha))\right|\nonumber\le \left|\mu_x^+(\hat B_{k}(x,n_2,\alpha))-u_\xi\frac{\fg_\varrho(k)\mu(B(\alpha))}{\sqrt{n_2}}\right|+\nonumber\\&\left| \mu_{\bar x}^+(\hat B_{\Psi (k)}(\bar x,n_2,\alpha))-u_\xi\frac{\fg_\varrho(\Psi(k)) \mu(B(\alpha))}{\sqrt{n_2}}\right|+|\fg_\varrho(k) -\fg_\varrho(\Psi(k))|\frac{u_\xi\cdot\mu(B(\alpha))}{\sqrt{n_2}}\nonumber\\&\le\frac{\epsilon^{10 a}}{\sqrt{n_2}}\mu(B(\alpha))+\frac{\epsilon^{10 a}}{\sqrt{n_2}}\mu(B(\alpha))+\frac{\epsilon^{10a}u_\xi\cdot\mu(B(\alpha))}{\sqrt{n_2}}\le \frac{3\epsilon^{10a}}{\sqrt{n_2}}\mu(B(\alpha)).
\end{align}

Now since $\phi$ depends only on the past, $S_{n_2}(\phi)(x^-,y_1)=S_{n_2}(\phi)(x^-,y_2)$ as long as $(\sigma^j(y_1))_0=(\sigma^j(y_2))_0$ for any $1\le j\le n_2$. This means that for each $k\in I$, $\hat B_{k}(x,n_2,\alpha)$ (and also $\hat B_{k}(\bar x,n_2,\alpha)$) is a union of cylinders of the form $$\mathscr C[\omega_0,\omega_1,\omega_2,\cdots,\omega_{n_2}]\subset \Sigma_A^+(x)\;\text{ ($\Sigma_A^+(\bar x)$ respectively)}.$$ Let $\prod_\alpha:=\prod_\alpha(k,x,n_2,[b_i^2/k_0^2])$ be the set of cylinders of the form $$\mathscr C[\omega_{n_2+1},\cdots,\omega_{[b_i^2/k_0^2]}]$$ such that for any cylinder $\mathscr C\subset\hat B_{k}(x,n_2,\alpha)$, $\mathscr C=\bigcup_{\mathscr C_1\in\prod_\alpha}(\mathscr C\cap \mathscr C_1)$. In particular (since $\hat B_{k}(x,n_2,\alpha)$ is a union of cylinders), 
\be\label{eq:zzjj}
\hat B_{k}(x,n_2,\alpha)=\bigcup_{\mathscr C_1\in\prod_\alpha}(\hat B_{k}(x,n_2,\alpha)\cap \mathscr C_1).
\ee


\begin{lemma}\label{le:m-k-a}
For any cylinder $\mathscr C\subset\hat B_{k}(x,n_2,\alpha)$, and $\widehat{\mathscr C}\subset\hat B_{\Psi(k)}(\bar x,n_2,\alpha)$, there is an $\epsilon/4 $ measure preserving map $\Phi_{k,\alpha,\mathscr C}:\mathscr C\to \widehat{\mathscr C}$ that maps $\mathscr C\cap \mathscr C_1$ to $\widehat{\mathscr C}\cap \mathscr C_1$ for any $\mathscr C_1\in\prod_\alpha$.
\end{lemma}
\begin{proof}
Notice that if $\omega_1\in\mathscr C$ and $\omega_2\in \widehat{\mathscr C}$, then it follows by the definition of $\hat B_{k}(x,n_2,\alpha)$ that $$\mathcal D_2(\sigma^{n_2}\omega_1,\sigma^{n_2}\omega_2)\le 2^{-[\epsilon^{-a}]}.$$Therefore by \eqref{eq:e-a}, there is an $\epsilon/4$ measure preserving map (given by the holonomy map) that maps $\sigma^{n_2}(\mathscr C\cap \mathscr C_1)$ to $\sigma^{n_2}(\widehat{\mathscr C}\cap \mathscr C_1)$ for any $\mathscr C_1\in\prod_\alpha$. By  invariance of the conditional measure, it follows that there is an $\epsilon/4$ measure preserving map $\Phi_{k,\alpha,\mathscr C}$ that maps $\mathscr C\cap \mathscr C_1$ to $\widehat{\mathscr C}\cap \mathscr C_1$ for any $\mathscr C_1\in\prod_\alpha$.
\end{proof}

Now combining \eqref{eee-1}, \eqref{eq:zzjj} and the above lemma, 
by discarding a set of measure at most $\frac{3\epsilon^{10a}}{\sqrt {n_2}}\mu(B(\alpha))$ from the set $\hat B_{k}(x,n_2,\alpha)$ or $\hat B_{\Phi(k)}(\bar x,n_2,\alpha)$, we get an $\epsilon/4$ measure preserving map $\Phi_{k,\alpha}$ from $\hat B_{k}(x,n_2,\alpha)$ to $\hat B_{\Psi(k)}(\bar x,n_2,\alpha)$ that maps 
 $\mathscr C\cap \mathscr C_1$ to $\widehat{\mathscr C}\cap \mathscr C_1$ for any $\mathscr C_1\in\prod_\alpha$.

%
%

 Combining all $\{\Phi_{k,\alpha}\}_{k\in I}$, we thus obtain an $\epsilon/4$ measure preserving map $\Phi_\alpha$ from $\mathbf B_\alpha$ to $\mathbf{\bar B}_\alpha$, by possibly discarding a subset of measure (from \eqref{eee-1})  \be\label{eq-es-i}\#I\cdot \frac{3\epsilon^{10a}}{\sqrt{n_2}}\mu(B(\alpha))\le 60k_0\epsilon^{10a}/\xi\mu(B(\alpha))\le \epsilon^{5a}\mu(B(\alpha)),\ee here we uses the fact that $\epsilon^{3a}\le \xi$ and $k_0\le \epsilon^{-a}$.

Moreover, by the definition of $\Phi_{\alpha}$ (more precisely $\Phi_{k,\alpha}$),  it follows that for every $y\in B_{k}(x,n_2,\alpha)\cap  \mathscr C_1$, we have that $\Phi_{k,\alpha}(y)\in B_{\Psi(k)}(x,n_2,\alpha)\cap  \mathscr C_1$ for any $\mathscr C_1\in\prod_\alpha$, which, by the definiton of $\Psi$ (see \eqref{eq:Psi}) implies that (recall that $\phi$ depends only on the past)
\begin{equation}\label{eq:sm2}
S_{n_2}(\phi)(x^-,y)-S_{n_2}(\phi)(\bar x^-,\Phi_\alpha(y))=\ell+\theta,
\end{equation}
for some $\theta:=\theta_{y,\Phi_\alpha(y)}\in (-\frac{\xi}{20},\frac{\xi}{20})$.
Moreover, for every $\ell\in [n_2,[b_i^2/k_0^2]]$,
\begin{equation}\label{eq:new2}
y_{\ell}=\Phi(y)_\ell.
\end{equation}

Varying $\alpha$, by possibly discarding a set of total measures at most (from \eqref{eq-be} and \eqref{eq-es-i}) $4\epsilon^{2}+2\epsilon^{5a}\leq \epsilon/2$, we can now glue the maps $\{\Phi_\alpha\}_{\alpha}$, to obtain an $\epsilon/2$ measure preserving map $\bar{\Phi}:\Sigma_A^+(x)\times\{z\}\to \Sigma_A^+(\bar x)\times\{\bar{z}\}$.  We now intersect the sets $\Sigma_A^+(x)\times\{z\}$ and  $\Sigma_A^+(\bar x)\times\{\bar{z}\}$ with $S_1$ (see the definition of $\Gamma_1$). By further discarding the set of points in $(\Sigma_A^+(x)\times \{z\})\cap S_1$ and $(\Sigma_A^+(\bar x)\times\{\bar{z}\})\cap S_1$ for which \eqref{clt-k} holds, and noticing the measure estimate in \eqref{clt-k} and the definition of $\Gamma_1$, we can restrict $\bar{\Phi}$ to obtain an $\epsilon$-measure preserving map 
\begin{equation}\label{eq:psi}\Phi:\Sigma_A^+(x)\times\{z\}\to \Sigma_A^+(\bar x)\times\{\bar{z}\}.
\end{equation}
  We have the following:

\begin{lemma}\label{lee-1} 
For any $n\in [n_2,[b_i^2/k_0^2]]$, 
\begin{equation}\label{S-est}
S_n(\phi)(x^-,y)-S_n(\phi)(\bar x^-,\Phi(y))=S_{n_2}(\phi)(x^-,y)-S_{n_2}(\phi)(\bar x^-,\Phi(y))+O(\epsilon^{4a}).
\end{equation}
\end{lemma}
\begin{proof}
Notice first that by cocycle identity, for any $n\ge n_2$, 
\begin{align}\label{E-sum}
&S_n(\phi)(x^-,y)-S_n(\phi)(\bar x^-,\Phi(y))-(S_{n_2}(\phi)(x^-,y)-S_{n_2}(\phi)(\bar x^-,\Phi(y)))\nonumber\\&=S_{n-n_2}(\phi)(\sigma^{n_2}(x^-,y))-S_{n-n_2}(\phi)(\sigma^{n_2}(\bar x^-,\Phi(y)))\nonumber\\&=\sum_{j=1}^{n-n_2}\left[\phi(\sigma^{n_2+j}(x^-,y))-\phi(\sigma^{n_2+j}(\bar x^-,\Phi(y)))\right].
\end{align}
Since $\phi$ is H\"older, there are a constant $C_2>0$ and a $\beta>0$ such that for any $\omega_1,\omega_2\in \Sigma_A$, $$|\phi(\omega_1)-\phi(\omega_2)|\le C_2[\cD_2(\omega_1,\omega_2)]^{\beta}.$$
Notice the definition of $\Phi$ (in particular \eqref{eq:new2}) (and the fact that  $\sigma_A^{n_2}(y),\sigma_A^{n_2}(\Phi y)\in B(\alpha)$), we have 
\begin{align}\label{eq:th}
\left|\phi(\sigma^{n_2+j}(x^-,y))-\phi(\sigma^{n_2+j}(\bar x^-,\Phi(y)))\right|&\le C_2[\cD_2(\sigma^{n_2+j}(x^-,y),\sigma^{n_2+j}(\bar x^-,\Phi(y)))]^{\beta}\nonumber\\
&\le C_2 2^{-(j+[\epsilon^{-a}])\beta}.
\end{align}
Therefore by \eqref{eq:th}, the absolute value of \eqref{E-sum} is bounded from above by $$C_2\sum_{j=1}^{n-n_2} 2^{-(j+[\epsilon^{-a}])\beta}\le C_22^{-[\epsilon^{-a}]\beta}\sum_{j=1}^{\infty} 2^{-j\beta}=C_22^{-[\epsilon^{-a}]\beta}\ll \epsilon^{4a}.$$This finishes the proof of \eqref{S-est}.
\end{proof}

\subsection{Concluding the proof}

We are now able to conclude the proof of Theorem \ref{main-1} by applying Proposition \ref{vwb}.

For every $\epsilon>0$, let $\hat n:=[b_i^2/k_0^2-1]$ and $G:=\{(s,z_r):r=\Sigma_A^+(s)\times \{z_r\}\in\Gamma_1\}$ (by \eqref{gamma-1}, $\mu(G)\ge 1-\epsilon$). Let $\Phi:\Sigma^+_A(x)\times\{z\}\to \Sigma^+_A({\bar{x}})\times{\bar{z}}$ be the $\epsilon$- measure preserving map constructed above (see \eqref{eq:psi}). Let $U\subset \Sigma^+_A(x)\times\{z\}$ be the set on which $\Phi$ is defined. Since $\Phi$ is $\epsilon$-measure preserving it follows that 
(1) in Proposition \ref{vwb} holds.

\begin{lemma}
For any $y\in \Sigma_A^+(x)$ (for which $\Phi$ is defined),
$$\#\bigg\{j\in[1,\hat n]:\begin{array}{l}T^j(x^-,y,z)\text{ and }T^j(\bar x^-,\Phi_{(x,z),(\bar x,\bar z)}(y,z))\\\text{are in the same atom of}\;\mathcal P_m\times \mathcal Q.\end{array}\bigg\}\ge (1-\epsilon)\hat n.$$
\end{lemma}
\begin{proof} Notice that by  \eqref{eq:sm2} and Lemma \ref{lee-1} for every $j\in [n_2,\hat n]$,
$$
\Big|(S_j(\phi)(x^-,y)(z)-S_j(\phi)(\bar{x}^-,\Phi(y))(\bar{z}))-\ell\Big|\leq \theta+O(\epsilon^{4a})
$$
Therefore, for some $\bar{\theta}:=\theta+O(\epsilon^{4a}) <\xi/10$,
$$
K_{S_j(\phi)(\bar x^-,\Phi(y))-\bar{\theta}}(\bar z)=K_{S_j(\phi)(x^-,y)(z)-\ell}(\bar z).
$$
So for every $j\in [n_2,\hat n]$ for which $K_{S_j(\phi)(\bar x^-,\Phi(y))}(\bar z)\in Z$, by \eqref{eq:cz},
\begin{equation}\label{eq:tr}
d\Big(K_{S_j(\phi)(\bar x^-,\Phi(y))}(\bar z),K_{S_j(\phi)(x^-,y)(z)-\ell}(\bar z)\Big)<c/20.
\end{equation}
Moreover, since by definition $y\in \Sigma_A^+(x)$ satisfies the estimate in \eqref{clt-k}, it follows that $|S_j(\phi)(x^-,y)(z)|\leq k_0\sqrt{j}\leq k_0\sqrt{\hat n}\leq k_0(b_i/k_0)\leq b_i$. Therefore, by \eqref{eq:zz} and (\textbf{C}) in the definition of quasi-elliptic, for every $j\in [n_2,\hat n]$ for which $K_{S_j(\phi)(x^-,y)(z)}(z)\in Z$
$$
d\Big(K_{S_j(\phi)(x^-,y)(z)}(z),K_{S_j(\phi)(x^-,y)(z)-\ell}(\bar z)\Big)<c/100.
$$
This together with \eqref{eq:tr} implies that for every $j\in [n_2,\hat n]$ for which $K_{S_j(\phi)(\bar x^-,\Phi(y))}(\bar z)\in Z$ and $K_{S_j(\phi)(x^-,y)(z)}(z)\in Z$, we have 
$$
d\Big(K_{S_j(\phi)(x^-,y)(z)}(z),K_{S_j(\phi)(x^-,\Phi(y))(\bar{z})}(\bar z)\Big)<c/10.
$$
This implies that for every $j\in [n_2,\hat n]$ for which $K_{S_j(\phi)(\bar x^-,\Phi(y))}(\bar z)\in Z\setminus V_c(\partial \mathcal Q)$ and $K_{S_j(\phi)(x^-,y)(z)}(z)\in Z$, 
$$
K_{S_j(\phi)(x^-,y)(z)}(z)\;\;\text{ and }\;\;K_{S_j(\phi)(x^-,\Phi(y))(\bar{z})}(\bar z) \text{ are in one atom of }\mathcal Q.
$$
By the definition of $\Phi$, $(x^-,\Phi(y),\bar{z})\in S_1$ and so by 
\eqref{E1} it follows that the proportion of such $j\in [n_2,\hat n]$ is at least $(1-\epsilon) \hat n$ (we also use that, by \eqref{eq:aa}, we have $n_2\le c^{20}\epsilon^{8}\hat n$).
Moroever, by \eqref{eq:new2}, for every $j\in [n_2,\hat n]$, $\sigma_A^j(y)$ is in the same atom of $\mathcal P_m$ as $\sigma_A^j(\Phi(y))$. Therefore, for  at least $(1-\epsilon)\hat n$ proportion of $j\in [n_2,\hat n]$ 
$$
T^j(x^-,y,z)\text{ and } T^j(\bar{x}^-,\Phi(y),\bar{z})\text{ are in one atom of  } \mathcal P_m\times \mathcal Q.
$$
This finishes the proof.
\end{proof}


This lemma finishes the proof of (2) in Proposition \ref{vwb}, hence completing the proof of Theorem \ref{main-1}.

\section{Quasi-elliptic flows}\label{sec:qe}

In this section we show that some natural flows are quasi-elliptic and each one of them has a regular generating partition. Let $T:(Y,\nu,d)\to (Y,\nu,d)$ be an ergodic automorphism and let $\psi\in L^1_+(\nu)$. Let $T^\psi$ denote the corresponding {\em special flow}, ie the flow on 
$$Y^\psi:=\{(x,s)\;:\; x\in Y, s<\psi(x)\}$$
given by 
$$
T^\psi(x,s):=(T^nx,s+t-S_n(\psi)(x)),
$$
where $n\in \Z$ is unique such that $S_n(\psi)(x)\leq t+s<S_{n+1}(\psi)(x)$. The flow preserves the measure $\nu^\psi$ which is the product measure restricted to $Y^\psi$. Moreover, let $\tilde{d}$ be the product metric on $Y^\psi$. Then  it is easy to see that (\textbf{B}) is satisfied for $T^\psi$ (with the metric $\tilde{d}$) and 
$$
Z_j:=\{(x,s)\;:\; \epsilon_j<s<\psi(x)-\epsilon_j\}.
$$
Assume now that there exists a sequence of {\em towers} $\mathcal{T}_n:=\bigcup_{t<h_n} T^\psi_tB_n$, where $B_n\subset Y$, $\sup_{t<h_n}diam(T^\psi_tB_n)\to 0$ and $T^\psi_tB_n\cap T^\psi_{t'}B_n=\emptyset$ for $0\leq t<t'<h_n$. Assume moreover that $\frac{\nu(B_n\cap T^\psi_{h_n}(B_n))}{\nu(B_n)}\to 1$.
\begin{lemma}Under the above assumptions, the flow $T^\psi$ is quasi-elliptic.
\end{lemma}
\begin{proof} As mentioned above, $(\textbf{C})$ holds. Let $(M_n)$ be a sequence tending to $+\infty$
such that $\nu(\tilde{B}_n)\to 1$, where  $$\tilde{B}_n:=\cap_{i=0}^{M_n}T^{\psi}_{-ih_n}(B_n).
$$  Let $\delta_i:=1/2\max_{t\leq h_i} diam(T_tB_i)$, $a_i=2h_i$ and $b_i=1/2M_ih_i$ and define
$$
N_i:=\Big(\bigcup_{t<h_i} T^\psi_t\tilde{B}_i\Big)\cap \Big\{(x,s): \delta_i\leq s\leq \psi(x)-\delta_i\}
$$
Notice that $\nu(N_i)\to 1$ as $i\to +\infty$.
Take any $(y,s),(y',s')\in N_i$. Then $(y,s)\in T^\psi_{t_1}\tilde{B}_i$ and $(y',s')\in T^\psi_{t_2}\tilde{B}_i$.  If $0<t_1\leq t_2<h_i$ we set $t_{y_1,y_2}:=t_2-t_1\in [0,h_i]$, otherwise, i.e. if $t_1>t_2$, we set $t_{y_1,y_2}=t_2-t_1+h_i\in [0,2h_i]$. It then follows by the definition of $\tilde{B}_i$ that $(y,s)$ and $T^\psi_{t_{y_1,y_2}}(y',s')$ are in one level of $\Big(\bigcup_{t<h_i} T^\psi_t\tilde{B}_i\Big)$ and (by the definition of $\delta_i$), $d((y,s),K_{-t_{y_1,y_2}}(y',s'))<\delta_i$ (the metric is just the product metric, see the definition of $N_i$). This gives (\textbf{A}). For (\textbf{C}) notice that 
$$
\bigcup_{t<M_ih_i}T^\psi(\tilde{B}_n)\subset \bigcup_{t<h_i}T^\psi(B_n).
$$
By the above, $(y,s)$ and $(y',s')$ are in one level of $T^\psi_{\bar{t}}(\tilde{B}_n)$ for some $\bar{t}<2h_i$. Therefore, for every $t_0<b_i=1/2M_ih_i$, we have that  $T^\psi_{t_0}(y,s)$ and $T^\psi_{t_0}(y',s')$ are in one level of $T^\psi_tB_n$ for some $t<h_i$. So if $T^\psi_{t}(y,s)\in Z_{j_0}$  for $t<b_i$(where $j_0$ is the smallest such that $\nu^\psi(Z_{j_0})\geq 1-\epsilon^2$), then for sufficiently large $i$, 
$$
d\Big(T^\psi_{t}(y,s), T^\psi_{t-t_{y_1,y_2}}(y,s)\Big)<\epsilon.
$$
So we get that (\textbf{C}) holds. This finishes the proof.
\end{proof}
Let $\mathcal S$ be a {\em translation flow} on a surface of genus $g$. Then it has a special representation over an interval exchange transformation and piecewise constant roof function $\psi$. It follows by a result of Katok, \cite{Kat}, that almost every (in the measure theoretic sense) such special flow is {\em rigid} with a sequence of towers $\mathcal{T}_n$, $\nu^\psi(\mathcal{T}_n)\to 1$.
and the base of the tower $B_n$ being an interval. In this case rigidity implies that $\frac{\nu(B_n\cap T^\psi_{h_n}(B_n))}{\nu(B_n)}\to 1$. It then follows from the above lemma that a.e. translation flow is quasi-elliptic.   We recall that by \cite{Avila-Forni}, a.e. translation flow is weakly mixing. 
 Analogously, if $T$ is an irrational rotation by $\alpha$ and $\psi$ is a $C^1$ function, then it follows that the corresponding flow $T^\psi$ is also rigid and analogously to the case of translation flows, quasi elliptic. Such flows arise as representations of reparametrizations of linear flows on the two dimensional torus.  Moreover, if the rotation is sufficiently Liouvillean, then  \cite{Fayad}, the corresponding special flow is also weakly mixing (for some roof functions $\psi$).

We also have the following lemma:
\begin{lemma}Let $T:(Y,\nu,d)\to (Y,\nu,d)$ be an ergodic automorphism such that there exists a regular generating partition\footnote{Recall that this means that $\nu(V_\eta(\partial \mathcal Q))\to 0$ as $\eta \to 0$.} for $T$. Let $\psi$ be a piecewise Lipshitz function. Then the corresponding special flow $T^\psi$ has a regular generating partition (see Definition \ref{Gen}).
\end{lemma}
\begin{proof}Let $\epsilon>0$ be small enough (depending on $\psi$ and $T$). By intersecting  $\mathcal Q=(Q_1,\ldots, Q_k)$ with balls of radius $\epsilon/2$ we can assume that $\max diam(Q_i)<\epsilon$. Let then $\mathcal Q^\psi$ be the partition of $Y^\psi$ where we partition each set 
$$
\{(x,s)\;:\; x\in Q_i\},
$$
into rectangles of height $\epsilon$, until we hit the graph, in which case this neighborhood of the graph is one additional atom. Let $t_0<\epsilon^2$ be such that $T^\psi_{t_0}$ is ergodic (recall that an ergodic flow can have at most countably many non-ergodic times). Take $(y,s)$ and $(y',s')$ and let $n\in \Z$ be such that $T^ny$ and $T^ny'$ are in different atoms of $\mathcal Q$. Assume WLOG that $n>0$ (if not, we argue in the negative direction).
Let $m=m(n)$ be the smallest such that $d(T^\psi_{mt_0}(y,s),(T^ny,0))<\epsilon$ (such $m$ exists by the definition of $t_0$). If for every $0\leq k< m$ the points $T^\psi_{kt_0}(y,s)$ and $T^\psi_{kt_0}(y',s')$ are in one atom of $\mathcal Q^\psi$, then it follows that if the first coordinate of $T^\psi_{kt_0}(y,s)$ is $T^{i_k}y$, then the first coordinate of $T^\psi_{kt_0}(y',s')$ is $T^{i_k}y'$. Indeed if $k_0<m$ was the smallest for which the first coordinates of $T^\psi_{k_0t_0}(y,s)$ and $T^\psi_{k_0t_0}(y',s')$ are different, then  $T^\psi_{(k_0-1)t_0}(y,s)$ and $T^\psi_{(k_0-1)t_0}(y',s')$ would have to be $\epsilon$ close to the graph of $\psi$ (since they are in one atom). But then  one of the points $T^\psi_{k_0t_0}(y,s)$, $T^\psi_{k_0t_0}(y',s')$ is still close to the graph, while the other is close to the base (since the first coordinates are different).
Then $T^\psi_{(m-1)t_0}(y,s)$ is a point which is $\epsilon$ close to the graph of $\psi$ and its first coordinate is $T^{n-1}y$ so the first coordinate of $T^\psi_{(m-1)t_0}(y',s')$ is $T^{n-1}y'$. Then The first coordinate of  $T^\psi_{mt_0}(y',s')$ is either $T^{n-1}y'$ and $T^\psi_{mt_0}(y',s')$ is still close to the graph of $\psi$ or it is $T^ny'$. In both cases $T^\psi_{mt_0}(y,s)$ and $T^\psi_{mt_0}(y',s')$ are in different atoms of $\mathcal{Q}^\psi$. Hence $\mathcal{Q}^\psi$ is a generating partition. Moreover $\mathcal{Q}^\psi$ is regular, since $\psi$ is piecewise Lipschitz. This finishes the proof.
\end{proof}

Notice that if $T$ is a {\em minimal} IET (including irrational rotations), then there always exist a regular generating parition for $T$. Indeed it is enough to take $\mathcal{P}$ to consist of sufficiently small intervals compatible with the continuity intervals for $T$. By minimality we then get that  for ANY two points $x,y\in [0,1)$ there exists a time $n_0$ such that $T^{n_0}x$ and $T^{n_0}y$ are in different atoms of $\mathcal{P}$.
As a consequence, we get:
\begin{corollary}A typical (in the measure theoretic sense) translation flow admits a regular generating partition and is quasi-elliptic. The same holds for special flows over rotations and under $C^1$ smooth roof functions.
\end{corollary}
\begin{remark}
We remark that it should be possible to prove quasi-ellipticity and existence of regular partitions for a broad class of {\em rank one} systems. In this paper we focused however on examples coming from smooth (or piecewise smooth) dynamics.
\end{remark}

\end{document}